\documentclass[12pt,reqno]{amsart}
\usepackage{amsmath, amsthm, amssymb, stmaryrd}
\usepackage{hyperref}
\usepackage{enumerate}
\usepackage{url}

\let\OLDthebibliography\thebibliography
\renewcommand\thebibliography[1]{
  \OLDthebibliography{#1}
  \setlength{\parskip}{0pt}
  \setlength{\itemsep}{0pt plus 0.3ex}
}

\usepackage{mathtools}

\usepackage{multirow, array}
\usepackage{placeins}

\advance \topmargin by -\headheight
\advance \topmargin by -\headsep
     
\evensidemargin \oddsidemargin
\newtheorem{thm}{Theorem}[section]
\newtheorem{lemma}[thm]{Lemma}

\theoremstyle{definition}

\theoremstyle{remark}

\numberwithin{equation}{section}

\newcommand{\mmod}[1]{{\,\,\mathrm{mod}\,\,#1}}

\newcommand*\wrapletters[1]{\wr@pletters#1\@nil}
\def\wr@pletters#1#2\@nil{#1\allowbreak\if&#2&\else\wr@pletters#2\@nil\fi}

\def\alp{{\alpha}} 
\def\bet{{\beta}}  
\def\gam{{\gamma}}

\def\tet{{\theta}}

\def\eps{\varepsilon}

\def\le{\leqslant} \def\ge{\geqslant}

\def \bN {\mathbb N}

\def \cA {\mathcal A}
\def \cB {\mathcal B}
\def \cC {\mathcal C}

\def \cE {\mathcal E}

\def \cG {\mathcal G}

\def \cI {\mathcal I}

\def \cL {\mathcal L}
\def \cM {\mathcal M}

\def \cP {\mathcal P}

\def \cS {\mathcal S}

\begin{document}
\title[Triangles with prime hypotenuse]{Triangles with prime hypotenuse}

\author[Sam Chow]{Sam Chow}
\address{The Mathematical Sciences Research Institute, 17 Gauss Way, Berkeley, CA 94720-5070, USA;
Department of Mathematics, University of York, Heslington, York, YO10 5DD, UK}
\email{sam.chow@york.ac.uk}

\author[Carl Pomerance]{Carl Pomerance}
\address{The Mathematical Sciences Research Institute, 17 Gauss Way, Berkeley, CA 94720-5070, USA;
Department of Mathematics, Dartmouth College, Hanover, NH 03755, USA}
\email{carl.pomerance@dartmouth.edu}

\thanks{}
\date{}
\begin{abstract} 
The sequence $3, 5, 9, 11, 15, 19, 21, 25, 29, 35,\dots$ consists of 
odd legs in right triangles with integer side lengths and prime hypotenuse. We show that the upper density of this sequence is zero, with logarithmic decay. The same estimate holds for the sequence of even legs in such
triangles.  We expect our upper bound, which involves the Erd\H{o}s--Ford--Tenenbaum constant, to be sharp
up to a double-logarithmic factor. We also provide a nontrivial lower bound.
Our techniques involve sieve methods, the distribution of Gaussian primes in narrow sectors,
 and the Hardy--Ramanujan inequality.
\end{abstract}
\maketitle

\section{Introduction}
\label{intro}

The sequence OEIS A281505 concerns odd legs in right triangles with integer side lengths and prime hypotenuse. By the parametrisation of Pythagorean triples, these are positive integers of the form $x^2 - y^2$, where $x,y \in \bN$ and $x^2 + y^2$ is prime. Even legs are those of the form $2xy$, where $x, y \in \bN$ and $x^2 + y^2$ is an odd prime. 
Let $\cA$ be the set of odd legs, and $\cB$ the set of even legs that occur in such triangles. Consider the quantities
\[
\cA(N) = \{ n \in \cA: n \le N \}, \qquad \cB(N) = \{ n \in \cB: n \le N \}
\]
as $N \to \infty$. 

Let $\cP$ denote the set of primes. By a change of variables, observe that
\[
\cA(N) = \# \{ ab \le N: \frac12 (a^2+b^2) \in \cP \}.
\] 
Additionally, note that
\[
\cB(2N) = \cC(N),
\]
where
\[
\cC(N) = \# \{1< ab \le N: a^2+b^2 \in \cP \}.
\]
We estimate $\cC(N)$, which is equivalent to estimating $\cB(N)$ and similar to estimating $\cA(N)$.

Let
\[
\eta = 1 - \frac{1 + \log \log 2}{\log 2} \approx 0.086
\]
be the Erd\H{o}s--Ford--Tenenbaum constant. This constant is related to the number of distinct products in the multiplication table, and also arises in other contexts, for example, see \cite{For2008}, \cite{FLP}, and \cite{MPP}.

\begin{thm} \label{UpperBound}
We have 
\[
\cC(N) \ll \frac{N}{(\log N)^\eta} (\log \log N)^{O(1)}.
\]
\end{thm}

Since every prime $p\equiv1\pmod4$ is representable as $a^2+b^2$ with $a,b$ integral,
we have $\cC(N)$ unbounded.  In fact, using the maximal order of the divisor function, we have
$\cC(N) \ge N^{1-o(1)}$ as $N\to\infty$.
We obtain a strengthening of this lower bound.

\begin{thm} \label{LowerBound}
We have, as $N\to\infty$,
\[
\cC(N) \ge\frac{N}{(\log N)^{\log4-1+o(1)}}.
\]
\end{thm}

Note that $\log4-1\approx 0.386$. Since $\cB(2N) = \cC(N)$, we obtain the same bounds for $\cB(N)$. By essentially the same proofs, one can also deduce these bounds for $\cA(N)$.

To motivate the outcome, consider the following heuristic. There are typically $\approx(\log n)^{\log 2}$ divisors of $n$, which follows from the normal number of prime factors of $n$, a result of Hardy and Ramanujan \cite{HR}. 
Moreover, given a factorisation $n=ab$, the ``probability"
 of $a^2+b^2$ being prime is roughly $(\log n)^{-1}$. Since $\log 2 < 1$, we expect the proportion $\cC(N)/N$ to decay logarithmically. In the presence of biases and competing heuristics, this \emph{prima facie} prediction should be taken with a few grains of salt. We use Brun's sieve and the Hardy--Ramanujan inequality to formally establish our bounds.  In addition, for Theorem \ref{LowerBound} we use a result of Harman and Lewis \cite{HL2001}
 on the distribution of Gaussian primes in narrow sectors of the complex plane.

We write $\cP$ for the set of primes. We use Vinogradov and Landau notation. As usual, we write $\omega(n)$ for the number of distinct prime divisors of $n$, and $\Omega(n)$ for the number of prime divisors of $n$ counted with multiplicity. The symbols $p$ and $\ell$ are reserved for primes, and $N$ denotes a large positive real number. 

\section*{Acknowledgments and a dedication}
The authors were supported by the National Science Foundation under Grant No. DMS-1440140 while in residence at the Mathematical Sciences Research Institute in Berkeley, California, during the Spring 2017 semester. The authors thank  John Friedlander and Roger Heath-Brown for helpful comments and
Tomasz Ordowski for suggesting the problem.

This year (2017) is the 100th anniversary of the publication of the 
paper {\it On the normal number of prime factors of
a number $n$}, by Hardy and Ramanujan, see \cite{HR}.
Though not presented in such terms, their paper ushered in the subject of probabilistic number theory.  
Simpler proofs have been found, but the original
proof contains a very useful inequality, one which we are happy to use yet again.  We dedicate
this note to this seminal paper.

\section{An upper bound}
\label{UpperSection}

In this section, we establish Theorem \ref{UpperBound}. The Hardy--Ramanujan inequality \cite{HR}
states that there exists a positive constant $c_0$ such that uniformly for $i \in \bN$  and $N\ge3$ we have
\[
\# \{n \le N: \omega(n) = i \} \ll \frac N {\log N} 		
\frac{(\log \log N + c_0)^{i-1}}{(i-1)!}.
\]	
By Mertens's theorem and the fact that the sum of the reciprocals of prime powers higher than the
first power converges, there is a positive constant $c_1$ such that
\begin{equation}
\label{eq:c1}
\sum_{p^\nu\le N}p^{-\nu}\le\log\log N+c_1 \qquad (N \ge 3).
\end{equation}

Let $\alp$ be a parameter in the range $1 < \alp < 2$, to be specified in due course. We begin by bounding the size of the exceptional set
\[
\cE_1 := \{ n \le N: \omega(n) > L \},
\]
where 
\begin{equation} \label{Ldef}
L = \lfloor \alp \log \log N \rfloor.
\end{equation}
By the Hardy--Ramanujan inequality, we have
\[
\# \cE_1
\ll \frac{N}{\log N} \sum_{i > L} \frac{(k+c_0)^{i-1}}{(i-1)!}= \frac{N}{\log N}\sum_{j\ge L}\frac{(k+c_0)^j}{j!},
\]
where $k= \log \log N$, and therefore
\[
\frac{\log N}N \# \cE_1 \ll \frac{(k+c_0)^{L}} {L!}
<\left(\frac{(k+c_0)e}{L}\right)^L=\left(\frac e\alp+O\left(\frac1k\right)\right)^L.
\]
Note that we have used here the elementary inequality $1/L!<(e/L)^L$, which holds for all positive
integers $L$ and follows instantly from the Taylor series for $e^L$.
Thus,
\begin{equation} \label{FirstTerm}
\# \cE_1 \ll \frac{N }{(\log N)^{1-\alp + \alp \log \alp}}.
\end{equation}

For an integer $n\ge2$,  write $P(n)$ for the largest prime factor of $n$, and let $P(1)=1$. By de Bruijn \cite[Eq. (1.6)]{deB} we may bound the size of the exceptional set
\[
\cE_2 := \{ n \le N: P(n) \le N^{1/\log \log N}\}
\]
by $N/(\log N)^2$ for all sufficiently large numbers $N$.  (Actually, the denominator may be taken as any fixed
power of $\log N$.)

Next, we estimate 
\[
\cC^*(N):= \# \{ ab \le N: ab \notin (\cE_1 \cup \cE_2),~ a^2+b^2 \in \cP \}.
\]
For $n$ counted by $\cC^*(N)$, we see by symmetry that we have $n = ab_0 \ell$ for some $a,b_0, \ell \in \bN$ with $\ell > N^{1/\log\log N}$ prime and $a^2 + b_0^2 \ell^2$ prime. Thus
\begin{equation} \label{pause}
\cC^*(N) \le 2 \sum_{\substack{{ab_0 \le N^{1- 1/\log \log N} }\\ \omega(ab_0) \le L}} S(a, b_0),
\end{equation}
where 
\[
S(a,b_0) = \sum_{\substack{ N^{1/\log \log N} < \ell \le \frac N{ab_0} \\ \ell,\;a^2 + b_0^2 \ell^2 \in \cP}}1.
\]

We turn our attention to $S(a,b_0)$. We may assume that $ab_0$ is even and $\gcd(a,b_0) = 1$, for otherwise $S(a,b_0)= 0$. Observe that
\[
S(a, b_0) \le \# \{ m \in (z,X]: \gcd(m(a^2+b_0^2 m^2), P(z)) = 1 \},
\]
where
\[
z = N^{(\log \log N)^{-3}}, \quad P(z) = \prod_{p < z} p, \quad X = \frac N{ab_0}.
\]
To bound this from above, we apply Brun's sieve \cite[Corollary 6.2]{FI2010} with
\[
\cA = \Bigl \{ m(a^2+b_0^2 m^2): 1 \le  m \le X \Bigr \},
\]
and with the completely multiplicative density function $g$ defined by 
\[
g(p) = \begin{cases}
1/p, & \text{if } p \mid ab_0 \text{ or } p \not \equiv 1 \mmod 4  \\
3/p, & \text{if } p \nmid a b_0, ~p \equiv 1 \mmod 4.
\end{cases}
\]

For this to be valid, we need to check that
\begin{equation} \label{BrunHyp}
|r_d(\cA)| \le g(d) d \qquad (d \mid P(z)),
\end{equation}
where
\[
r_d(\cA) = |\cA_d| - Xg(d), \quad \cA_d = \{ n \in \cA: n \equiv 0 \mmod d \}
\]
and $P(z) = \prod_{p < z} p$. We begin by noting that if $p \in \cP$ then the congruence
\[
m (a^2 + b_0^2 m^2) \equiv 0 \mmod p
\]
has $g(p)p$ solutions $m \mmod p$. Observe that any divisor $d$ of $P(z)$ must be squarefree; thus, by the Chinese remainder theorem, the congruence 
\[
m (a^2 + b_0^2 m^2) \equiv 0 \mmod d
\]
has $g(d)d$ solutions $m \mmod d$. By periodicity, we now have
\[
r_d(\cA) = \# \{ m \le M: m(a^2 + b_0^2 m^2) \equiv 0 \mmod d \} - Mg(d),
\]
where $M = X - d \lfloor X/d \rfloor$. This confirms \eqref{BrunHyp}, since $0 \le M < d$ and $0 < g(d) \le 1$.

We also need to check that
\[
\log z \le \frac{\log X}{c \log ( V(z)^{-1} \log X)},
\]
where $V(z) = \prod_{p < z} (1- g(p))$, and where
\[
(c/e)^c = e, \qquad c \approx 3.59.
\]
This follows from the inequalities
\[
X \ge N^{1/\log \log N}, \qquad V(z) \gg (\log z)^{-2}.
\]

Now \cite[Corollary 6.2]{FI2010} tells us that
\[
S(a,b_0) \le X^{3/4} + 2XV(z) \ll \frac{N(\log \log N)^{O(1)}}{(\log N)^2 ab_0}.
\]
(Note that we might equally have used the version of Brun's sieve in \cite[p. 68]{HRi}, which
is less precise, but somewhat easier to utilise.)
Substituting this into \eqref{pause} yields
\begin{equation} \label{subI}
\cC^*(N) \le \frac{N(\log \log N)^{O(1)}}{(\log N)^2} I,
\end{equation}
where
\[
I = \sum_{j+k \le L} \sum_{\substack{a \le N \\ \omega(a)=j}}  a^{-1}  \sum_{\substack{b_0 \le N \\ \omega(b_0)=k}}  b_0^{-1}.
\]

It follows from the multinomial theorem that
\begin{align*}
I &\le \sum_{j+k \le L} j!^{-1} \Bigl(\sum_{p^v \le N} p^{-v}\Bigr)^j
k!^{-1} \Bigl(\sum_{p^v \le N} p^{-v}\Bigr)^k \\
&= 
\sum_{j+k \le L} (j+k)!^{-1} {j+k \choose j} \Bigl( \sum_{p^v \le N} p^{-v} \Bigr)^{j+k}.
\end{align*}
Letting $m=j+k$, the binomial theorem now gives
\[
I \le \sum_{m  \le L} m!^{-1} \Bigl( 2 \sum_{p^v \le N} p^{-v} \Bigr)^m \le \sum_{m \le L} \frac{(2 \log \log N + 2c_1)^m}{m!},
\]
where $c_1$ is as in \eqref{eq:c1}. In view of \eqref{Ldef}, we now have
\begin{align*}
I& \ll L!^{-1} (2 \log \log N + 2c_1)^L<\left(\frac{2e\log\log N+2ec_1}{L}\right)^L\\
&=\left(\frac{2e}{\alp}+O\left(\frac1L\right)\right)^L\ll(\log N)^{\alp(1+\log2- \log \alp)}.
\end{align*}
Substituting this into \eqref{subI} yields
\begin{equation} \label{CstarBound}
\cC^*(N) \le N(\log \log N)^{O(1)} (\log N)^{\alp(1+\log 2 - \log \alp) - 2}.
\end{equation}

By \eqref{FirstTerm}, our estimate for $\#\cE_2$, and \eqref{CstarBound}, we have
\[
\cC(N) \le \cC^*(N) + \# \cE_1 + \# \cE_2 \le N (\log \log N)^{O(1)} (\log N)^{-\cM},
\]
where 
\[
\cM = \min \{ 1 - \alp + \alp \log \alp, ~ 2 ,~ 2 - \alp - \alp \log 2 + \alp \log \alp \}.
\]
We now choose  $1 < \alp < 2$ so as to maximise $\cM$.
One might guess that this $\alp$ solves
\[
1- \alp + \alp \log \alp = 2 - \alp - \alp \log 2 + \alp \log \alp,
\]
and indeed $\alp = (\log 2)^{-1}$ does maximise $\cM$ on the interval $(1,2)$. 
With this choice of $\alp$, we have
\[
\cM = 1 - \frac{1 + \log \log 2}{\log 2} = \eta,
\]
completing the proof of Theorem \ref{UpperBound}.

\section{A lower bound}
\label{LowerSection}

In this section, we establish Theorem \ref{LowerBound}. Let
\[
\cL_0 = \{ (a,b) \in \bN^2:1< ab \le N,~ a^2 + b^2 \in \cP \}.
\]
Writing $P(n)$ for the largest prime factor of $n>1$, and $P(1) = 1$, put
\[
\cL_1 = \{(a,b) \in \cL_0: P(ab) \le N^{1/\log \log N} \}.
\]
Let $\eps$ be a small positive real number, and let
\begin{align*}
\cL_2 &= \{(a,b) \in \cL_0 \setminus \cL_1: \omega(a) > (1+\eps) \log \log N \}, \\
\cL_3 &= \{(a,b) \in \cL_0 \setminus \cL_1: \omega(b) > (1+\eps) \log \log N \}.
\end{align*}
Finally, write
\[
\cL = \cL_0 \setminus (\cL_1 \cup \cL_2 \cup \cL_3).
\]
As we seek a lower bound, we are free to discard some inconvenient elements of $\cC(N)$. Thus, by the Cauchy--Schwarz inequality, we have
\begin{equation} \label{Cauchy}
\cC(N) \ge (\#\cL)^2 / \cS(N),
\end{equation}
where $\cS(N)$ is the number of quadruples $(a,b,c,d) \in \bN^4$ such that
\[
ab=cd\hbox{ and }(a,b),(c,d)\in\cL.
\]

We first show that
\begin{equation} \label{L0bound}
\# \cL_0 \gg N.
\end{equation}
For this, we use existing work counting Gaussian primes in narrow sectors. For convenience, we state the relevant result \cite[Theorem 2]{HL2001}. 

\begin{thm}[Harman--Lewis]
\label{HL}
Let $X$ be a large positive real number, and let $\bet, \gam$ be real numbers in the ranges
\[
0 \le \bet \le \pi/2, \qquad X^{-0.381} \le \gam \le \pi/2.
\]
Then 
\[
\# \{ (a,b) \in \bN^2: a^2 + b^2 \in \cP \cap [0,X], ~\arctan(b/a) \in [\beta, \beta + \gam) \} \gg \frac{\gam X}{\log X}.
\]
The implied constant is absolute.
\end{thm}

For positive integers $i \le \frac{ \log N}{10 \log 2}$, we apply this with 
\[
\bet = \gam = \frac \pi {2^{i+1}}, \qquad X = 2^{i-2}N.
\]
By Jordan's inequality
\[
\frac 2\pi x \le \sin x \le x \qquad (0 \le x \le \pi/2),
\]
observe that if $a,b \in \bN$, $a^2 + b^2 \le X$ and $\tet = \arctan(b/a) \le \pi 2^{-i}$ then
\[
ab \le X \sin \tet \cos \tet = \frac12 X \sin (2 \tet) \le X \tet \le N2^{i-2} \cdot \frac \pi {2^i} \le N.
\]
Thus
\[
\# \cL_0 \gg \sum_{i \le \frac{\log N}{10 \log 2}} \frac N{\log N} \gg N,
\]
confirming \eqref{L0bound}.

Next, we show that $\# \cL_j = o(N)$ ($j=1,2,3$).

\begin{lemma} \label{discard0}
We have $\# \cL_1 = o(N)$.
\end{lemma}

\begin{proof}
By de Bruijn \cite[Eq. (1.6)]{deB}, we have
\[
\sum_{a \le \sqrt N} \sum_{\substack{b \le N/a \\ P(b) \le N^{1/\log \log N}}} \ll \sum_{a \le \sqrt N} \frac N{a (\log N)^2} \ll \frac N{\log N}.
\]
Thus, by symmetry, we have $\# \cL_1 \ll \frac N{\log N}$.
\end{proof}

\begin{lemma} \label{discard} We have
\[
\# \cL_j = o(N) \qquad (j = 2,3).
\]
\end{lemma}

\begin{proof}
As $\# \cL_2 = \# \cL_3$, we need only show this for $j=2$. Taking out a prime factor $\ell > N^{1/\log \log N}$ of $ab$, we have
\[
\# \cL_2 \le 2\sum_{\substack{a \le N^{1-1/\log \log N} \\ \omega(a) > (1+\eps) \log \log N } } \sum_{b \le a^{-1} N^{1- 1/\log \log N}} S_{a,b},
\]
where
\[
S_{a,b} =  \sum_{\substack{N^{1/\log \log N} < \ell \le \frac N{ab} \\ \ell,\: a^2 + b^2 \ell^2 \in \cP}}1.
\]
As in the last section, Brun's sieve implies that
\[
S_{a,b} 
\ll \frac{N(\log \log N)^{O(1)}}{ab (\log N)^2}.
\]

Therefore
\begin{equation}
\label{lem3.3eq}
\# \cL_2 \ll  \frac {N(\log \log N)^{O(1)}} {\log N} \sum_{\substack{a \le N^{1-1/\log \log N}\\\omega(a) \ge T} }a^{-1},
\end{equation}
where
\begin{equation}
 \label{Tdef}
T = \lfloor (1+\eps) \log \log N \rfloor.
\end{equation}
As in the prior section, the multinomial theorem implies that
\begin{align*}
\sum_{\substack{a\le N^{1-1/\log\log N}\\\omega(a)\ge T}}\frac1a&
\le\sum_{j\ge T}\frac1{j!}\left(\log\log N+c_1\right)^j
\ll_\eps \frac1{T!}(\log\log N+c_1)^T\\
&\le \left(\frac{e\log\log N+ec_1}{T}\right)^T\ll(\log N)^{(1+\eps)(1-\log(1+\eps))}.
\end{align*}
Since $(1+\eps)(1-\log(1+\eps))<1$, 
using this estimate in \eqref{lem3.3eq} completes the proof of the lemma.
\end{proof}

Combining \eqref{L0bound} with Lemmas \ref{discard0} and \ref{discard} gives
\begin{equation} \label{Lbound}
\# \cL \gg N.
\end{equation}

\begin{lemma} \label{L2}
If $c' > \log 4 -1$ then
\[
\cS(N) \ll_{c'} N (\log N)^{c'}.
\]
\end{lemma}

\begin{proof}
One component of the count is when $(a,b)=(c,d)$.  This is the diagonal case, and it is
easily estimated.   By the sieve, the number of pairs $(a,b)\in\cL$ with $a\le b$ is at most
\[
\sum_{a\le\sqrt{N}}\sum_{b\le N^{1-1/\log\log N}/a}\sum_{\substack{\ell\le N/ab\\a^2+\ell^2b^2\in\cP}}1
\le \frac{N(\log\log N)^{O(1)}}{(\log N)^2}\sum_{a,b}\frac1{ab}\le N(\log\log N)^{O(1)},
\]
 which is negligible.  (Note that this estimate shows that  \eqref{Lbound} is essentially tight.)
 
   For the nondiagonal case we imitate 
 \S \ref{UpperSection}. If $(a,b,c,d)$ is counted by $\cS(N)$, put
\[
g =\gcd(a,c), \quad a = gu, \quad c= gv,
\]
so that
\[
ub = vd, \quad d = uw, \quad b = vw.
\]
Recall \eqref{Tdef}, and let $\cG$ be the set of $(g,u,v,w_0) \in \bN^4$ such that
\[
guvw_0 \le N^{1-1/ \log \log N}, \qquad \omega(gu), \omega(vw_0), \omega(gv), \omega(uw_0) \le T,\qquad u\ne v.
\] 
As $P(ab) > N^{1/ \log \log N}$, we see by symmetry that
\begin{equation} \label{cS1}
\cS(N) \ll N(\log\log N)^{O(1)} + \sum_{(g,u,v,w_0) \in \cG} S(g,u,v,w_0),
\end{equation}
where
\[
S(g, u ,v, w_0) = \sum_{\substack{\ell\in\cP,\,N^{1/\log \log N} < \ell \le \frac N{guvw_0} \\ (gu)^2 + (vw_0)^2 \ell^2, \;(gv)^2 + (uw_0)^2 \ell^2 \in \cP}} 1.
\]

The fact that $u \ne v$ ensures that there are three primality conditions defining $S(g,u,v,w_0)$. To bound $S(g,u,v,w_0)$ from above, we may assume without loss that $guvw_0$ is even, and that the variables $g,u,v,w_0$ are pairwise coprime, for otherwise $S(g,u,v,w_0) = 0$. 
Paralleling \S \ref{UpperSection}, an application of Brun's sieve reveals that
\begin{equation} \label{FinalBrun}
S(g, u ,v, w_0) \ll \frac{ N (\log \log N)^{O(1)}} {guvw_0 (\log N)^3}.
\end{equation}

Substituting \eqref{FinalBrun} into \eqref{cS1} yields
\begin{equation} \label{cS2}
\cS(N) \ll N(\log\log N)^{O(1)} + \frac{N (\log \log N)^{O(1)}}{(\log N)^3} \cI,
\end{equation}
where 
\[
\cI = \sum_{k_1 + \cdots + k_4 \le 2T} \prod_{i =1}^4 \Bigl(\sum_{n \le N:\, \omega(n) = k_i} n^{-1} \Bigr)
\]
and $T$ is as in \eqref{Tdef}.  
With $U = 2T$, it follows from the multinomial theorem that
\begin{align*}
\cI &\le \sum_{k_1 + \cdots + k_4 \le U} \prod_i k_i!^{-1} \Bigl(\sum_{p^v \le N} p^{-v}\Bigr)^{k_i}
\\
&=  \sum_{m  \le U} m!^{-1} \sum_{k_1 + \cdots + k_4 = m} { m \choose k_1, k_2, k_3, k_4 } 
\Bigl( \sum_{p^v \le N} p^{-v} \Bigr)^m,
\end{align*}
and a further application of the multinomial theorem gives
\[
\cI \le \sum_{m  \le U} m!^{-1} \Bigl(4 \sum_{p^v \le N} p^{-v} \Bigr)^m \le \sum_{m \le U} \frac{(4 \log \log N + 4c_1)^m}{m!}.
\]
As $U = 2(1+\eps)\log \log N+O(1)$, we now have
\begin{align*}
\cI &\ll \frac{ (4 \log \log N +4c_1)^U}{U!}<\left(\frac{4e\log\log N+4ec_1}{U}\right)^U\\
&=\left(\frac{4e}{2+2\eps}+O\left(\frac1U\right)\right)^U\ll(\log N)^{2(1+\eps)(1+\log2-\log(1+\eps))}.
\end{align*}
Substituting this into \eqref{cS2} yields
\begin{align*} 
\cS(N)& \ll N(\log \log N)^{O(1)} (\log N)^{2(1+\eps)(1+\log 2 - \log (1+\eps)) - 3}\nonumber\\
&\le N(\log\log N)^{O(1)}(\log N)^{\log4-1+2\eps(1+\log2)}.
\end{align*}
As $c' > \log 4 -1$, we may choose $\eps > 0$ to give
$\cS(N) \ll_{c'} N (\log N)^{c'}$.
\end{proof}

Combining \eqref{Cauchy} and \eqref{Lbound} with Lemma \ref{L2} establishes Theorem \ref{LowerBound}.

\section{A final comment}

We conjecture that Theorem \ref{UpperBound} holds with equality.   For a lower bound, one might
restrict attention to those pairs $(a,b)$ with $\omega(a)\approx\omega(b)\approx\frac1{2\log2}\log\log N$.
The upper bound for the second moment is analysed as in the paper, getting $N/(\log N)^{\eta+o(1)}$; we expect that a more refined analysis would give
\[
\frac {N (\log \log N)^{O(1)}} {(\log N)^\eta}
\]
here. The difficulty is in obtaining this same estimate as a lower bound for the first moment.
This  would follow if we had an analogue of Theorem \ref{HL} in which $a,b$ have a restricted number of prime
factors.  Such a result holds for the general distribution of Gaussian primes, at least if one restricts
only one of $a,b$, see \cite{FI}.

\providecommand{\bysame}{\leavevmode\hbox to3em{\hrulefill}\thinspace}

\end{document}